\documentclass[12pt]{article}
\usepackage{amssymb}
\usepackage{amscd}
\usepackage{amsmath}
\usepackage{amsthm}
\usepackage{amsfonts}
\usepackage{mathrsfs}
\newtheorem{theorem}{Theorem}

\newtheorem{lemma}[theorem]{Lemma}

\theoremstyle{definition}

\theoremstyle{remark}

\begin{document}

\title{\Large\bf On duplicate representations as $\boldsymbol{2^x+3^y}$ for nonnegative integers~$\boldsymbol{x}$ and~$\boldsymbol{y}$}

\author{\it Douglas Edward Iannucci}\date{}

\maketitle

\begin{abstract}
We prove a conjecture posted in the Online Encyclopedia of Integer Sequences, namely that there are exactly five positive integers that can be written in more than one way as the sum of a nonnegative power of~2 and a nonnegative power of~3. The case for both powers being positive follows from a theorem of Bennett. We use elementary methods to prove the case where zero exponents are allowed.\end{abstract}
\section{Introduction}
In the Online Encyclopedia of Integer Sequences ({\it OEIS\/}, Sloane~\cite{oeis}), sequence {\tt A004050} comprises the integers of the form $2^x+3^y$ for nonegative integers~$x$ and~$y$. On this sequence's entry in the {\it OEIS\/}, it was remarked as a conjecture in September 2012 that only five of these integers can be so expressed in two different ways. 

In fact, sequence {\tt A085634} lists those very integers representable both as $2^x+3^y$ and $2^a+3^b$, with $x$, $y$, $a$, and~$b$ nonnegative integers and $x>a$. The  five elements listed are
\begin{alignat*}{5}
5&=2^2+3^0,\quad 11&=2^3+3^1,\quad 17&=2^4+3^0,\quad 35&=2^5+3^1,\quad 259&=2^8+3^1,\\
 5&=2^1+3^1,\quad 11&=2^1+3^2,\quad 17&=2^3+3^2,\quad 35&=2^3+3^3,\quad  259&=2^4+3^5.
 \end{alignat*}
On the entry in the {\it OEIS\/} for sequence {\tt A085634}, it was remarked in February 2005 that if~$n$ is in the sequence and $n>259$, then $n>10^{4000}$. In this note, we render this lower bound vacuously true by proving the conjecture: indeed, the five numbers listed above are the only elements of {\tt A085634}.  
 
We thus assume that
\begin{equation}\label{scaduto}
2^x+3^y=2^a+3^b,
\end{equation}
where~$x$, $y$, $a$, and~$b$ are nonnegative integers, such that (without loss of generality) $x>a$ (whence $y<b$). 

Equivalently,
\begin{equation}\label{bennett}
2^x-3^b=2^a-3^y.
\end{equation}
This brings us to sequence {\tt A207079} in the {\it OEIS\/}, which is described in its entry as ``the only nonunique differences between powers of 3 and 2.'' It is given as a finite sequence of five elements, namely 1, 5, 7, 13, and~23. It is commented that the finiteness of this sequence is due to Bennett~\cite{bennett}, who, in fact, proved a more general result, from which the finiteness of {\tt A207079} follows directly. In his article, he gives a clear, precise history of the general problem of determining the number of solutions to the exponential Diophantine equation $|a^x-b^y|=c$, and we learn that the finiteness of the specific sequence {\tt A207079} was first proved in 1982. We state here, as a lemma, the special case of Bennett's result that applies most directly to~\eqref{bennett}.

\begin{lemma}\label{thmbennett}
{\bf(Bennett)} There are precisely three integers of the form $2^x-3^b$, with~$x$ and~$b$ natural numbers, that are also expressible as $2^a-3^y$, with~$a$ and~$y$ natural numbers such that $x> a$. They are
$$-1=2^3-3^2=2-3,\qquad 5=2^5-3^3=2^3-3,\qquad 13=2^8-3^5=2^4-3.$$
These are, respectively, the only two such representations for these three integers. All other integers have either a unique such representation, or none at all. 
\end{lemma}

We apply Bennett's result to the cases of~\eqref{scaduto} and~\eqref{bennett} where~$x$, $y$, $a$, and~$b$ are all positive integers. This leaves us with the special case when $y=0$; clearly \eqref{scaduto} and~\eqref{bennett} are impossible if $a=0$. We prove the special case $y=0$ by elementary methods, except for the one instance where we apply Lemma~\ref{thmbennett} to deduce that~1 has only the single representation $1=2^2-3$ (although it is not difficult to prove this fact independently).

\section{The case when $\boldsymbol{y>0}$}

\begin{theorem}\label{y_one}
There are precisely three solutions to~\eqref{scaduto} when $y>0$. They are
$$11=2^3+3=2+3^2,\qquad 35=2^5+3=2^3+3^3,\qquad 259=2^8+3=2^4+3^5.$$
\end{theorem}

\begin{proof} 
Let $c=2^a-3^y$ in~\eqref{bennett}. By Lemma~\ref{thmbennett}, if $c\notin\{-1,5,13\}$, then~$x=a$, which contradicts the hypothesis $x>a$. Otherwise, $c\in\{-1,5,13\}$. 

Suppose $c=-1$. By Lemma~\ref{thmbennett}, we have the two representations, as in~\eqref{bennett},
$$-1=2^3-3^2=2-3.$$
Thus, $x=3$, $b=2$, $a=1$, and $y=1$. This produces
$$2^3+3=2+3^2=11.$$

Suppose $c=5$. Similarly, 
$$5=2^5-3^3=2^3-3,$$
thus producing
$$2^5+3=2^3+3^3=35.$$

Suppose $c=13$. Similarly,
$$13=2^8-3^5=2^4-3$$
produces
$$2^8+3=2^4+3^5=259.$$
\end{proof}

\section{The case when $\boldsymbol{y=0}$}
For a prime~$p$ and a natural number~$n$, we write $p\|n$ if $p\mid n$ but $p^2\nmid n$. We denote the {\it $p$-valuation of~$n$} by $v_p(n)$: i.e., $v_p(n)=k$ if $p^k\|n$.

\begin{lemma}\label{v2}
 If~$n$ is a natural number then
$$v_2(3^n-1)=\begin{cases}
1,&\text{if $2\nmid n$;}\\
2+v_2(n),&\text{if $2\mid n$.}\end{cases}$$
\end{lemma}
\begin{lemma}\label{v3} 
 If~$n$ is a natural number then
$$v_3(2^n-1)=\begin{cases}
0,&\text{if $2\nmid n$;}\\
1+v_3(n),&\text{if $2\mid n$.}\end{cases}$$
\end{lemma}
Lemmata~\ref{v2} and~\ref{v3} follow easily from Theorems~94 and~95, Nagell~\cite{nagell}.

\begin{theorem}\label{y_zero}
There are precisely two solutions to~\eqref{scaduto} when $y=0$. They are
$$5=2^2+1=2+3,\qquad 17=2^4+1=2^3+3^2.$$
\end{theorem}

\begin{proof} 
We are given
\begin{equation}\label{case1conj}
2^x+1=2^a+3^b,
\end{equation}
where~$x$, $a$, and~$b$ are natural numbers, where $x>a$. Let $s=x-a$. Thus,
\begin{equation}\label{case1fac}
2^a(2^s-1)=3^b-1.
\end{equation}
It is necessary by Lemma~\ref{v3} that~$s$ is odd, as, by~\eqref{case1fac}, $3\nmid2^s-1$. 

First, suppose~$b$ is odd. Then Lemma~\ref{v2} implies $2\|3^b-1$, hence, by~\eqref{case1fac}, $a=1$. Thus, by~\eqref{case1conj},
$$2^x-3^b=1.$$
Thus, by Lemma~\ref{thmbennett}, $x=2$ and $b=1$. This produces the equation 
$$2^2+1=2+3=5.$$

It remains to let~$b$ be even. Then $a=2+v_2(b)$ by Lemma~\ref{v2}. Suppose $2^2\mid b$. Then $3^4-1\mid3^b-1$, hence $5\mid3^b-1$. Then~\eqref{case1fac} implies $5\mid2^s-1$, hence $4\mid s$, a contradiction as~$s$ is odd. Therefore $2\|b$ and $a=3$. Writing $b=2c$ for an odd natural number~$c$, we have by~\eqref{case1fac}
$$2^s-1=\frac{3^c-1}2\cdot\frac{3^c+1}4.$$
Letting
$$z=\frac{3^c+1}4,$$
then~$z$ is a natural number by Lemma~\ref{v2}, and we obtain the quadratic in~$z$,
$$2^s-1=2z^2-z.$$
Completing the square yields
$$(4z-1)^2=2^{s+3}-7.$$
Writing $s=2t+1$ yields the difference of squares factorization
$$ (2^{t+2}-4z+1)(2^{t+2}+4z-1)=7.$$
Therefore
$$2^{t+2}-4z+1=1,\qquad 2^{t+2}+4z-1=7;$$
thus,
$$2^{t+2}=4z=4.$$
Therefore $t=0$, $z=1$; thus, $c=1$. Hence $s=1$ and $b=2$. Recalling $a=3$, we have~$x=4$. This produces the equation 
$$2^4+1=2^3+3^2=17.$$
\end{proof}

\vskip 24pt\noindent
\it Director, Center for Numerology\\ 
\it  University of the Virgin Islands\\ 
\it 2 John Brewers Bay \\ 
\it St. Thomas VI 00802\\ 
\it USA\\
\tt diannuc@uvi.edu\\ 
\tt diannuc@gmail.com\\


\begin{thebibliography}{9}

\bibitem{bennett}
M.~A.~Bennett, ``Pillai's conjecture revisited,'' {\it J. Number Theory\/} {\bf 98} (2003) 228--235.


\bibitem{nagell}
T. Nagell,
\emph{Introduction to Number Theory,\/}
Wiley Publishers, New York, 1951.

\bibitem{oeis} 
N.~J.~A.~Sloane, The On-Line Encyclopedia of Integer
Sequences,  {\tt http://oeis.org}. 

\end{thebibliography}
\end{document}